\documentclass[12pt, a4paper,reqno]{amsart}
\usepackage{amsthm}
{\theoremstyle{definition}
\newtheorem{dfn}{Definition}[section]}
\newtheorem{prop}[dfn]{Proposition}

\newtheorem{thm}[dfn]{Theorem}
{\theoremstyle{definition}
\newtheorem{rem}[dfn]{Remark}}
\newtheorem{lem}[dfn]{Lemma}
\newtheorem{cor}[dfn]{Corollary}

\newtheorem{ques}[dfn]{Question}
{\theoremstyle{definition}
\newtheorem{exa}[dfn]{Example}}

\usepackage[top=30truemm,bottom=30truemm,left=35truemm,right=35truemm]{geometry}
\usepackage{amsmath,amssymb,amscd,bm,graphicx,tikz-cd,upgreek,dsfont,amsfonts,tikz,enumitem}
\usetikzlibrary{cd}
\usepackage[all]{xy}
\usepackage[colorlinks=true,
    citecolor=green(pigment),
    linkcolor=alizarin,
    urlcolor=denim]{hyperref}
    
\definecolor{alizarin}{rgb}{0.82, 0.1, 0.26}
\definecolor{azure(colorwheel)}{rgb}{0.0, 0.5, 1.0}
\definecolor{blue(pigment)}{rgb}{0.2, 0.2, 0.6}
\definecolor{denim}{rgb}{0.08, 0.38, 0.74}
\definecolor{mint}{rgb}{0.24, 0.71, 0.54}
\definecolor{parisgreen}{rgb}{0.31, 0.78, 0.47}
\definecolor{persiangreen}{rgb}{0.0, 0.65, 0.58}
\definecolor{seagreen}{rgb}{0.18, 0.55, 0.34}
\definecolor{shamrockgreen}{rgb}{0.0, 0.62, 0.38}
\definecolor{green(pigment)}{rgb}{0.0, 0.65, 0.31}

\usepackage{setspace}
\newcommand{\marginparstretch}{0.6}
\let\oldmarginpar\marginpar
\renewcommand\marginpar[1]{\-\oldmarginpar[\framebox{\setstretch{\marginparstretch}\begin{minipage}{\marginparwidth}{\raggedleft\tiny #1}\end{minipage}}]{\framebox{\setstretch{\marginparstretch}\begin{minipage}{\marginparwidth}{\raggedright\tiny #1}\end{minipage}}}}

\usepackage{scalerel,stackengine}
\stackMath
\newcommand\reallywidecheck[1]{%
\savestack{\tmpbox}{\stretchto{%
   \scaleto{%
    \scalerel*[\widthof{\ensuremath{#1}}]{\kern-.6pt\bigwedge\kern-.6pt}%
    {\rule[-\textheight/2]{1ex}{\textheight}}%WIDTH-LIMITED BIG WEDGE
   }{\textheight}% 
}{0.5ex}}%
\stackon[1pt]{#1}{\scalebox{-1}{\tmpbox}}%
}
\def\ang<#1>{\langle #1 \rangle}
\def\bigang<#1>{\left\langle #1 \right\rangle}

\numberwithin{equation}{section}

\newcommand{\Spec}{\operatorname{Spec}}
\newcommand{\Supp}{\operatorname{Supp}}
\newcommand{\Hom}{\operatorname{Hom}}
\newcommand{\Ext}{\operatorname{Ext}}

\newcommand{\Pic}{\operatorname{Pic}}
\newcommand{\Aut}{\operatorname{Aut}}

\newcommand{\Max}{\operatorname{Max}}

\newcommand{\coh}{\operatorname{coh}}

\newcommand{\Db}{\operatorname{D^b}}

\newcommand{\Perf}{\operatorname{Perf}}

\newcommand{\Ker}{\operatorname{Ker}}

\newcommand{\Th}{\mathbf{Th}}

\def\Ser{\mathop{\sf Ser}\nolimits}

\def\Z{\mathop{\sf Z}\nolimits}

\newcommand{\cC}{\mathcal{C}}

\newcommand{\cE}{\mathcal{E}}
\newcommand{\cF}{\mathcal{F}}

\newcommand{\cH}{\mathcal{H}}
\newcommand{\cI}{\mathcal{I}}

\newcommand{\cL}{\mathcal{L}}
\newcommand{\cM}{\mathcal{M}}
\newcommand{\cN}{\mathcal{N}}
\newcommand{\cO}{\mathcal{O}}
\newcommand{\cP}{\mathcal{P}}
\newcommand{\cQ}{\mathcal{Q}}

\newcommand{\cS}{\mathcal{S}}
\newcommand{\cT}{\mathcal{T}}
\newcommand{\cU}{\mathcal{U}}

\newcommand{\cZ}{\mathcal{Z}}

\newcommand{\bC}{\mathbb{C}}

\newcommand{\bP}{\mathbb{P}}

\newcommand{\bR}{\mathbb{R}}
\newcommand{\bS}{\mathbb{S}}

\newcommand{\bZ}{\mathbb{Z}}

\newcommand{\simto}{\xrightarrow{\sim}}
\newcommand{\hookto}{\hookrightarrow}

\renewcommand{\l}{\langle}
\renewcommand{\r}{\rangle}

\newcommand{\spec}{\mathrm{ Spec}_{\triangle}}

\newcommand{\sspec}{\mathrm{ Spec}_{\triangle}^{\sf Ser}}

\newcommand{\bspecx}{\mathrm{ Spec}_{\otimes_X}}

\newcommand{\bspec}{\mathrm{ Spec}_{\otimes}}

\newcommand{\gwedge}{\raisebox{1pt}[0pt][0pt]{$\textstyle \bigwedge^{\raisebox{-1pt}[0pt][0pt]{\tiny\!$i$\,}}$}}

\tikzset{
        DB/.style={circle,draw=black,circle,fill=white,inner sep=0pt, minimum size=4pt},
        DW/.style={circle,draw=black,fill=black,inner sep=0pt, minimum size=4pt},
        cvertex/.style={circle,draw=black,fill=white,inner sep=1pt,outer sep=3pt},
        vertex/.style={circle,fill=black,inner sep=1pt,outer sep=3pt},
        star/.style={circle,fill=yellow,inner sep=0.75pt,outer sep=0.75pt},
        tvertex/.style={inner sep=1pt,font=\scriptsize},
	pvertex/.style={circle,inner sep=1pt,outer sep=2pt,font=\scriptsize},
        gap/.style={inner sep=0.5pt,fill=white}
}

\makeatletter
\newcommand*{\defeq}{\mathrel{\rlap{%
                     \raisebox{0.3ex}{$\m@th\cdot$}}%
                     \raisebox{-0.3ex}{$\m@th\cdot$}}%
                     =}
\makeatother

\begin{document}
\title[]{Prime thick subcategories\\ on elliptic curves}
\author[Y.~Hirano]{Yuki Hirano}
\author[G.~Ouchi]{Genki Ouchi}

\address{Y.~Hirano, Department of Mathematics, Kyoto University, Kitashirakawa-Oiwake-cho, Sakyo-ku, Kyoto, 606-8502, Japan}
\email{y.hirano@math.kyoto-u.ac.jp}

\address{G.~Ouchi, Graduate School of Mathematics, Nagoya University, Furocho, Chikusaku, Nagoya, Japan, 464-8602}
\email{genki.ouchi@math.nagoya-u.ac.jp}

\begin{abstract}
 We classify all prime thick subcategories in the derived category of coherent sheaves on elliptic curves, and determine the Serre invariant locus of Matsui spectrum of derived category of coherent sheaves on any smooth projective curves.   
\end{abstract}

\maketitle{}

\section{Introduction}

\subsection{Background} Classification of certain thick subcategories of a given  triangulated category  is a fundamental problem in the study of triangulated categories, and  many authors classified certain thick subcategories of the derived or the singularity categories of schemes \cite{thomason, stevenson, takahashi1, takahashi2,  mt, hirano,ks,el}. For a tensor triangulated category $(\cT,\otimes)$, Balmer introduced the notion of prime ideals of $(\cT,\otimes)$ which are certain thick subcategories of $\cT$, and he  defined a topological space $\Spec_{\otimes}\cT$ whose underlying set is  the set of prime ideals of $(\cT,\otimes)$. Moreover, in the case when $(\cT,\otimes)$ has a unit object $1\in \cT$,  he introduced a sheaf $\cO$ of commutative rings on $\Spec_{\otimes}\cT$   and  proves that for a noetherian scheme $X$, there is an isomorphism 
\[
X\simto \Spec_{\otimes_X}\Perf X
\]
of ringed spaces, where $\Perf X$ denotes the  perfect derived category of  $X$. This shows that the classification of certain class of  thick subcategories enables us to extract geometric information of a noetherian scheme $X$ from the triangulated category $\Perf X$. Recently Matsui introduced the notion of prime thick subcategories of an arbitrary triangulated category $\cT$, and he defined a topological space $\Spec_{\triangle}\cT$ whose underlying set is the set of prime thick subcategories of $\cT$ \cite{matsui}. We call this space $\spec\cT$ {\it Matsui spectrum} of $\cT$. In \cite{matsui},  he proved that for a noetherian scheme $X$ there is an immersion 
\[
X\hookto \Spec_{\triangle}\Perf X
\]
of topological spaces, and it is homeomorphism if $X$ is quasi-affine. In particular, if a smooth variety $X$ is a Fourier-Mukai partner of another smooth variety $Y$, i.e. if we have an exact equivalence $\Db(X)\cong\Db(Y)$ of bounded derived categories of coherent sheaves, the underlying topological space of $Y$ can be embedded into the space $\Spec_{\triangle}\Db(X)$. Since smooth projective complex  varieties of dimension greater than one are determined  only by their underlying Zariski topologies \cite{klos}, studying the topological space $\Spec_{\triangle}\Db(X)$ is important to study Fourier-Mukai partners of a smooth projective variety $X$. 

\subsection{Results}
In this paper, we determine the Matsui spectrum of the derived categories of coherent sheaves on elliptic curves. Let $X$ be an elliptic curve over an algebraically closed field $k$.
Consider the set 
\[I\defeq\{(r,d)\in \bZ^2 \mid r>0, \gcd(r,d)=1\}.\]
For an element $(r,d) \in I$, let $M(r,d)$ be the moduli space of $\mu$-semistable sheaves on $X$ with Chern character $(r,d)$. Then a universal family $\cU_{r,d}$ of $M(r,d)$ defines a Fourier-Mukai equivalence $\Phi_{\cU_{r,d}}\colon \Db(M(r,d))\simto\Db(X)$, and this induces an immersion
\[
M(r,d)\hookto \spec\Db(M(r,d))\xrightarrow{\Phi_{\cU_{r,d}}}\spec\Db(X).
\]
Denote by $M_{r,d}$ the image of the above immersion $M(r,d)\hookto \spec\Db(X)$. The following is the main theorem.

\begin{thm}[Theorem \ref{thm:elliptic curve}]\label{main intro}
We have the equality
\[\spec \Db(X)=\bspecx\Db(X) \sqcup \bigsqcup_{(r,d) \in I} M_{r,d}. \]
\end{thm}

We also introduce and discuss the Serre invariant locus of Matsui spectrum. For a triangulated category $\cT$ with a Serre functor $\bS$, we define a subspace
\[
\Spec^{\Ser}_{\triangle}\cT\subseteq\Spec_{\triangle}\cT
\]
of Matsui spectrum of $\cT$ to be the set of $\bS$-invariant prime thick subcateories of $\cT$.  This space is invariant under arbitrary exact equivalences of triangulated categories with Serre functors. By the results in \cite{matsui}, it is easy to see that if a Gorenstein projective variety $X$ has an ample or anti-ample canonical bundle, then we have a homeomorphism
\begin{equation}\label{top reconstruction}
\Spec^{\Ser}_{\triangle}\Perf X\cong X.
\end{equation}
This is a topological  reconstruction of a Gorenstein projective variety $X$ with ample or anti-ample canonical bundle from the triangulated category $\Perf X$. 
As an immediate consequence of Theorem \ref{main intro} and the homeomorphism \eqref{top reconstruction}, we obtain explicit descriptions of the Serre invairant loci $\sspec\Db(X)$ on any smooth projective curves $X$.

\subsection{Notation and Convention}
\begin{itemize}
    \item In this paper, we treat only essentially small triangulated categories over a field $k$. 
    \item Functors between triangulated categories that we consider are all  exact functors.
    \item For a proper morphism $f\colon X\to Y$ of smooth varieties, we write $f_*\colon \Db(X)\to \Db(Y)$ and  $f^*\colon \Db(Y)\to \Db(X)$ for the derived push-forward and the derived pull-back respectively. Similarly, we denote by $\otimes_X$ the derived tensor product on $X$.
    \item For a variety (or a scheme) $X$,  a point in $X$ is not necessarily a closed point. 
\end{itemize}

\subsection{Acknowledgements}
The authors would be grateful to  W.~Hara for helpful discussions, and H.~Matsui for Remark 4.15. They also would like to thank the referee for careful reading this article, and for helpful comments.   Y.H. is supported by JSPS KAKENHI Grant Number 19K14502. G.O. is supported by JSPS KAKENHI Grant Number 19K14520.

%For objects $E,F \in \Db(X)$, we define the Euler characteristic
%\[ \chi(E,F):=\sum_{i \in \bZ}(-1)^i\dim \Ext^i(E,F).\]
%It defines the biliner form on the Grothendieck group $K_0(X)$ of $X$.
%The numerical Grothendieck group $K_{\mathrm{num}}(X)$ is defined by
%\[K_{\mathrm{num}}(X):=K_0(X)/\mathrm{Ker}(\chi_0), \]
%where the linear map $\chi_0:K_0(X) \to \Hom_{\bZ}(K_0(X), \bZ)$ is defined by %$\chi_0([E]):=\chi([E],-)$. 

%For an object $E \in \Db(Z)$, we define the Mukai vector  $v(E) \in H^{\mathrm{ev}}(Z,\bQ)$ by

%\[v(E):= \mathrm{ch}(E)\sqrt{\mathrm{td}_Z} \in H^{\mathrm{ev}}(Z,\bQ).\]

%Then we have the cohomological Fourier-Mukai transform 
%\[\Phi^H_\cE: H^{\mathrm{ev}}(X,\bQ) \to H^{\mathrm{ev}}(Y,\bQ), \alpha \mapsto p_*(q^*(\alpha) \cdot v(\cE))\]
%associated to $\Phi_\cE$.

\section{Spectrum of prime thick subcategories}

Throughout this section $\cT$ denotes an essentially small triangulated category over a field $k$.

\subsection{Prime thick subcategories of triangulated categories}
Following \cite{matsui,mt}, we recall the basics of prime thick subcategories of triangulated categories.

A {\it thick subcategory} of $\cT$ is a full triangulated subcategory  of $\cT$ that is closed under taking direct summands. Let $\Th(\cT)$ be the set of all thick subcategories of $\cT$. The set $\Th(\cT)$ is a partially ordered set with respect to the inclusion relation.  
For a full subcategory $\cE$ of $\cT$, we define 
\[ \mathsf{Z}(\cE)\defeq\{\cI \in \Th(\cT) \mid \cI \cap \cE=\emptyset\}.\] 
Then it is easy to see the following (see \cite[Definition 2.1]{mt}). 
\begin{itemize}
    \item[$(1)$]We have  $\mathsf{Z}(\cT)=\emptyset$ and  $\mathsf{Z}(\emptyset)=\Th(\cT)$.
    \item[$(2)$]For a family $\{\cE_i\}_{i \in I}$ of full subcategories of $\cT$, we have  $\bigcap_{i \in I}\mathsf{Z}(\cE_i)=\mathsf{Z}(\bigcup_{i \in I}\cE_i)$.
    \item[$(3)$]For full subcategories $\cE_1, \cE_2$ of $\cT$, we have
    $\mathsf{Z}(\cE_1) \bigcup \cZ(\cE_2)=\mathsf{Z}(\cE_1 \oplus \cE_2)$, where $\cE_1 \oplus \cE_2\defeq \{E_1 \oplus E_2 \in \cT \mid E_1 \in \cE_1, E_2 \in \cE_2 \}$.
\end{itemize}
In particular, $\Th(\cT)$ has the topology whose closed subsets are of the form $\mathsf{Z}(\cE)$.

\begin{dfn}[{\cite[Definition 2.2]{matsui}}] A thick subcategory $\cP$ of $\cT$ is said to be {\it prime} if there exists a unique minimal element $\overline{\cP}$ in the partially orderd set $ \{\cQ\in\Th(\cT)\mid \cP\subsetneq\cQ\}$. We denote by  
\[
\Spec_{\triangle}\cT
\] 
the set of prime thick subcategories of $\cT$. The set $\Spec_{\triangle}\cT$ has the induced topology from $\Th(\cT)$, which we call the {\it Zariski topology}. We call this topological space the {\it Matsui Spectrum} of $\cT$.
\end{dfn}

An exact equivalence $\Phi\colon \cT\simto \cT'$ of triangulated categories induces a homeomorphism
\begin{equation}\label{homeo}
\Phi\colon \Spec_{\triangle}\cT\simto\Spec_{\triangle}\cT'
\end{equation}
that sends $\cP$ to  $\Phi(\cP)\subsetneq \cT'$. 
In particular, the autoequivalence group $\Aut\cT$ acts on the Matsui spectrum  $\Spec_{\triangle}\cT$.

Let $\cU\in \Th(\cT)$ be a thick subcategory, and write $\pi\colon\cT\to\cT/\cU$ for the natural quotient functor. Then the inverse images define a continuous map 
\[
\pi^{-1}\colon \Th(\cT/\cU)\to\Th(\cT).
\]
 A map $f\colon X\to Y$ between topological spaces is called an  {\it immersion} if it  induces a homeomorphsim $X\simto f(X)$.

\begin{prop}[{\cite[Proposition 2.9]{matsui}}]\label{verdier}
The map $\pi^{-1}$ restricts to an immersion
\[\pi^{-1}\colon \Spec_{\triangle}\cT/\cU\to\Spec_{\triangle}\cT,\] and its image is $\{\cP\in\Spec_{\triangle}\cT\mid \cU\subseteq\cP\}$.

\end{prop}

Let $\Phi\colon\cT\to \cT'$ be a fully faithful exact functor that admits a right adjoint functor $\Psi\colon \cT'\to \cT$. Then $\Psi$ induces an equivalence
\[
\cT'/\Ker(\Psi)\simto \cT.
\]
We denote by 
\begin{equation}\label{admissibe immersion}
    \Phi_*\defeq \Psi^{-1}\colon \spec\cT\hookto\spec\cT'
\end{equation}
the immersion induced from the quotient functor $\Psi$. If $\Phi$ is an equivalence, the immersion $\Phi_*$ is nothing but the homeomorphism $\Phi\colon \spec\cT\simto \spec\cT'$ as in \eqref{homeo}.

\subsection{Balmer spectrum} 
A {\it tensor structure} on $\cT$ is an associative symmetric bifunctor $\otimes\colon \cT\times \cT\to \cT$ that is exact in each variable. 
A {\it tensor triangulated category} $(\cT, \otimes)$ is  a  triangulated category $\cT$ with a tensor structure $\otimes$ on $\cT$. 

\begin{rem}
In this paper, we do not require the existence of a unit object in the definition of tensor triangulated categories.
\end{rem}

\begin{exa}The usual tensor products $\otimes_X$ on the perfect complexes $\Perf X$ on a noetherian scheme $X$ is a tensor structure on $\Perf X$. In particular, if $X$ is a regular scheme,  $\Db(X)$ admits a tensor structure $\otimes_X$. 
\end{exa}
%smooth?non-singular?

Let $(\cT,\otimes)$ be a tensor triangulated category. A thick subcategory $\cI \in \Th(\cT)$ is called an {\it ideal} of $(\cT, \otimes)$ if $A \otimes E \in \cI$ holds for any objects $A \in \cT$ and $E \in \cI$.

\begin{dfn}
An ideal $\cP\neq \cT$ of $(\cT, \otimes)$ is said to be  {\it prime}  if for objects $A,B \in \cT$, the condition $A \otimes B \in \cP$ implies $A\in \cP$ or $B \in \cP$.
The set of all prime ideals of $(\cT,\otimes)$ is denoted by 
\[
\Spec_{\otimes}\cT,
\] and it is called the {\it Balmer spectrum} of $(\cT,\otimes)$. 
We regard $\Spec_\otimes\cT$ as a topological subspace of $\Th(\cT)$.
\end{dfn}

Let $X$ be a noetherian scheme. For simplicity, we  will write  
\[
\Spec_{\otimes}\Perf X\defeq\Spec_{\otimes_X}\Perf X
\]
when no confusion arises.
For an object $F\in\Perf X$ we define the {\it support}
\[
\Supp F\defeq\{x\in X\mid F_x\not\cong 0 \mbox{ in }\Perf\cO_{X,x}\}
\]
Note that 
$\Supp F=\bigcup_{i\in \bZ}\Supp \cH^i(F)$. Since the support  of each sheaf cohomology $\cH^i(F)$  is closed in $X$, and $\cH^i(F)\neq0$ only for finitely many $i\in\bZ$, the support $\Supp F$ is closed in $X$.
\begin{dfn}
For $x \in X$, we define the thick subcategory  
\[
\cS_X(x)\defeq\{F \in \Perf X \mid x\notin \Supp(F)\}
\]
of $\Perf X$.
\end{dfn}

In \cite{balmer} Balmer reconstruct a noetherian scheme $X$ from the  tensor triangulated category $(\Perf X,\otimes_{X})$. In this paper, the following is enough for our purpose. 

\begin{prop}[{\cite[Corollary 5.6 ]{balmer}}]\label{prop Balmer}
For $x \in X$, the thick subcategory $\cS_X(x)\in\Th(\Perf X)$ is a prime ideal of $(\Perf X, \otimes_X)$. Moreover, the map
\[
\cS_X\colon X\simto \Spec_{\otimes}\Perf X
\]
is a homeomorphism.
\end{prop}

\begin{rem}\label{rem:supp}
Proposition \ref{prop Balmer} implies that a prime ideal $\cP$ of $(\Perf X,\otimes_X)$ does not contain an object $F\in\Perf X$ with $\Supp F=X$. In particular, $\cP$ does not contain any locally free sheaves. 
\end{rem}

\begin{thm}[{\cite[Corollary 3.8, Corollary 4.9]{matsui}}]\label{matsui thm}
We have the following:
\begin{itemize}
\item[$(1)$] 
Let $\cI$ be an ideal of $(\Perf X,\otimes_X)$. Then $\cI$ is a prime thick subcategory if and only if $\cI$ is a prime ideal. 
\item[$(2)$] We have an immersion
\[
\cS_X\colon X\hookto \Spec_{\triangle}\Perf X
\]
whose image is $\bspec\Perf X$, and it is a homeomorphism if $X$ is quasi-affine. 
\end{itemize}
\end{thm}

Let $i\colon U\hookto X$ be an open immersion. Write $Z\defeq X\backslash U$ for the complement of $U$, and define
\[
\Perf_ZX\defeq\{F\in \Perf X\mid \Supp(F)\subseteq Z\}.
\]
It is well known that the restriction functor 
\[
\pi_U\defeq(-)|_U\colon \Perf X\to \Perf U
\]
induces an equivalence between $\Perf U$ and the idempotent completion of the Verdier quotient $\Perf X/\Perf_ZX$  (see \cite[Theorem 2.13]{balmer2}). By \cite[Proposition 2.11]{matsui}, we have a homeomorphism
\begin{equation}\label{eq:up to direct}
\Th(\Perf X/\Perf_ZX)\cong\Th(\Perf U)
\end{equation}
that restricts to the homeomorphism
\[
\spec(\Perf X/\Perf_ZX)\cong\spec \Perf U.
\]
Thus by Proposition \ref{verdier} we have an immersion
\[
\pi_U^{-1}\colon\spec \Perf U\hookto \spec \Perf X.
\]
The following says that the immersion in Theorem \ref{matsui thm} commutes with open immersions.

\begin{lem}\label{open commute}
For any $x\in U$,  we have 
\[
\pi_U^{-1}(\cS_U(x))=\cS_X(x).
\]
In particular, the following diagram is commutative:
\[
\begin{tikzcd}
 \spec\Perf U\arrow[rr,"\pi_U^{-1}"]&&\spec\Perf X\\
 U\arrow[rr,"i"]\arrow[u,"\cS_U"]&&X.\arrow[u,"\cS_X"']
 \end{tikzcd}
\]
\begin{proof}
 For an object $F\in \Perf X$, 
\begin{eqnarray*}
F\in \pi_U^{-1}(\cS_U(x))&\Longleftrightarrow& \pi_U(F)=F|_U\in \cS_U(x)\\
&\Longleftrightarrow& x\notin \Supp(F|_U)\\
&\Longleftrightarrow& (F|_U)_x\cong 0
\end{eqnarray*}
Since $x\in U$, we have $F_x\cong (F|_U)_x$. Hence $F\in \pi_U^{-1}(\cS_U(x))$ if and only if $F\in \cS_X(x)$.
\end{proof}
\end{lem}

\begin{prop}\label{closure}
The closure of $\bspec \Perf X$ in $\spec\Perf X$ is $\Z(\cE)$, where $\cE\defeq\{E\in \Perf X\mid \Supp E=X\}$.
\begin{proof}
By Remark \ref{rem:supp}, we have an inclusion  $\bspec \Perf X\subseteq \Z(\cE)$. Assume that there is a  family $\cF$ of objects in $\Perf X$ such that $\bspec \Perf X\subseteq \Z(\cF)$. It is enough to show that $\cF\subseteq \cE$.   If the family $\cF$ contains an object $F$ such that $\Supp F\neq X$,  there is a point $x\in X$ with $x\notin \Supp F$. Then we obtain $F\in \cS_X(x)$, which contradicts to $\bspec \Perf X\subset \Z(\cF)$. 
\end{proof}
\end{prop}

%%%%%%%%%%%%%%%%%%%%%%%%%%%%%%%%%%%%%%%%%%%%%%%%%%%%%

\section{Prime thick subcategories on smooth varieties}

%%%%%%%%%%%%%%%%%%%%%%%%%%%%%%%%%%%%%

\subsection{Intersections of Balmer spectrums of smooth varieties}

Let $X$ and $Y$ be smooth  varieties, and   $\Phi \colon \Db(X)\hookto\Db(Y)$ a fully faithful functor that admits a right adjoint functor $\Psi\colon \Db(Y)\to \Db(X)$. Then we have an immersion
\[
\Phi_*\colon \spec\Db(X)\hookto\spec\Db(Y)
\]
defined as in \eqref{admissibe immersion}. In this section, we discuss sufficient conditions for the intersection $\Phi_*(\bspec\Db(X))\cap\bspec\Db(Y)$ to be  empty or non-empty respectively. For this, we first prove some lemmas.

\begin{lem}[cf. {\cite[Lemma 5.3]{bm}}]\label{supp lem}
Let $F\in \Db(X)$ be an object, and let $x\in X$ be a closed point. Then 
the following are equivalent:
\begin{itemize}
    \item[$(1)$] $x\notin\Supp(F)$
    \item[$(2)$] $\Hom_{\Db(X)}(k(x),F[l])=0$ for all $l\in\bZ$
    \item[$(3)$] $\Hom_{\Db(X)}(F,k(x)[l])=0$ for all $l\in \bZ$
\end{itemize} 
\begin{proof}
Denote by $j_x\colon \Spec k(x)\hookto X$ the closed immersion, and set $d\defeq \dim X$.  Since $j_x^*F\cong\bigoplus_{m\in\bZ}H^m(j_x^*F)[-m]$ in $\Db(k(x))$, the adjunctions $j_{x*}\dashv j_x^!$ and $j_x^*\dashv j_{x*}$ implies the  following isomorphisms
\begin{equation}\label{1st isom}
\Hom_{\Db(X)}(k(x),F[l])\cong \Hom_{\Db(k(x))}(k(x),j_x^*F[l-d])\cong H^{d-l}(j_x^*F)
\end{equation}
\begin{equation}\label{2nd isom}
\Hom_{\Db(X)}(F,k(x)[l])\cong\Hom_{\Db(k(x))}(j_x^*F,k(x)[l])\cong H^{-l}(j_x^*F)^{\vee}.
\end{equation}
By \cite[Lemma 3.3(a)]{thomason} the condition (1) holds if and only if $j_x^*F\cong 0$ in $\Db(k(x))$, or equivalently $H^m(j_x^*F)\cong 0$ for all $m\in \bZ$. Hence the result follows from \eqref{1st isom} and \eqref{2nd isom}.
\end{proof}
\end{lem}

\begin{lem}\label{closed points}
Notation is same as above. Let $x\in X$ and $y\in Y$ be closed points. Then $\Phi_*(\cS_X(x))=\cS_Y(y)$ if and only if $\Phi(k(x))\cong k(y)[l]$ for some $l\in \bZ$.
\begin{proof}
($\Rightarrow$) Assume that $\Phi_*(\cS_X(x))=\cS_Y(y)$. Let $y'\in Y$ be a closed point with $y'\neq y$. Then $k(y')\in \cS_Y(y)=\Phi_*(\cS_X(x))$, and so  $x\notin \Supp\Psi(k(y'))$. By Lemma \ref{supp lem},  for any  $l\in \bZ$ we have $\Hom_{\Db(X)}(k(x),\Psi(k(y'))[l])=0$. By the adjunction $\Phi\dashv \Psi$ and Lemma \ref{supp lem}, we see that $y'\notin \Supp\Phi(k(x))$. Hence we see that $\Supp\Phi(k(x))=\{y\}$. Since $\Phi$ is fully faithful, we have  \[\Hom(\Phi(k(x)),\Phi(k(x))[m])\cong \begin{cases} 
k(x)&  \mbox{ when } m=0\\
0& \mbox{ when } m<0
\end{cases}\] 
Hence by \cite[Lemma 4.5]{huybre} we obtain an isomorphism  $\Phi(k(x))\cong k(y)[l]$ for some $l\in\bZ$. \\
($\Leftarrow$) Assume that $\Phi(k(x))\cong k(y)[l]$ for some $l\in \bZ$. Then using Lemma \ref{supp lem} for $F\in \Db(Y)$ we see that 
\begin{eqnarray*}
F\in \cS_Y(y) &\Longleftrightarrow& \Hom_{\Db(Y)}(k(y), F[m])=0 \mbox{ for all } m\in\bZ\\
&\Longleftrightarrow& \Hom_{\Db(X)}(k(x), \Psi(F)[m])=0 \mbox{ for all } m\in\bZ\\
&\Longleftrightarrow& \Psi(F)\in \cS_X(x)\\
&\Longleftrightarrow& F\in \Psi^{-1}(\cS_X(x))=\Phi_*(\cS_X(x))
\end{eqnarray*}
Hence we obtain $\Phi_*(\cS_X(x))=\cS_Y(y)$.
\end{proof}
\end{lem}

%Recall that the {\it Fourier-Mukai functor} $\Phi_{\cP}\colon \Db(X) \to \Db(Y)$ associated to a Fourier-Mukai kernel $\cP \in \Db(X \times Y)$ is the exact functor defined by
%\[\Phi_\cP(E)\defeq p_*(q^*(E) \otimes \cP), \]
%where $p\colon X \times Y \to Y$ and $q\colon X \times Y \to X$ are natural projections. %It is well known that the Fourier-Mukai kernels
%\[
%\cP_R\defeq \cP^{\vee}\otimes q^*\omega_X[\dim X]\hspace{3mm}\mbox{ and }\hspace{3mm}
%\cP_L\defeq\cP^{\vee}\otimes p^*\omega_Y[\dim Y]
%\]
%defines a right adjoint functor and a left adjoint functor of $\Phi_{\cP}$ respectively.

\begin{lem}\label{lem:FM and supp}
Let $X$ and $Y$ be smooth  varieties such  that $\dim X>0$. Let  
$\Phi\colon \Db(X) \xrightarrow{\sim} \Db(Y)$ be an exact equivalence.
Assume that for any closed point $x \in X$, $\Supp (\Phi(k(x)))=Y$ holds. Then we have
\[\Phi(\Spec_{\otimes}\Db(X))\cap \Spec_{\otimes}\Db(Y)=\emptyset.\]
\end{lem}
\begin{proof}
Take a point $x \in X$.  
Since $\dim X>0$, there is a closed point $x' \in X$ such that $x'\neq x$. Then we have $k(x') \in \cS_X(x)$, and so $\Phi(k(x')) \in \Phi(\cS_X(x))$. For any point $y \in Y$, by Remark \ref{rem:supp} and the assumption, we have $\Phi(k(x')) \notin \cS_Y(y)$.  This implies that  $\Phi(\cS_X(x)) \neq \cS_Y(y)$. 
Therefore, $\Phi(\Spec_{\otimes}\Db(X))\cap \Spec_{\otimes}\Db(Y)=\emptyset$ holds.
\end{proof}

\begin{exa}\label{abelian variety}
Let $A$ be an abelian variety over $\bC$. Denote the dual abelian variety of $A$ by $A^\vee$.
Let $\cU$ be a Poincar\'e line bundle over $A^\vee \times A$. 
By \cite[Theorem 2.2]{muk}, we have an equivalence $\Phi_{\cU}\colon \Db(A^\vee) \xrightarrow{\sim} \Db(A)$.
For a closed point $[L] \in A^\vee$, we have $\Phi_{\cU}(k([L]))=L$. 
By Proposition \ref{prop Balmer}, Theorem \ref{matsui thm} and Lemma \ref{lem:FM and supp}, there is an immersion
\[A \sqcup A^\vee \xrightarrow{\sim} \bspec\Db(A) \sqcup \,\Phi_{\cU}\left(\bspec\Db(A^\vee)\right) \hookrightarrow \spec\Db(A).\]
 \end{exa}

\begin{exa}\label{K3}
Let $X$ be a complex K3 surface and $H$ an ample divisor on $X$. 
Take $v=(r,c,d) \in H^*(X,\bZ)$ such that $c \in \mathrm{NS}(X)$. Let $M_H(v)$ be the moduli space of $H$-semistable sheaves with Mukai vector $v$.  
Assume that $r\geq 0$ and $\gcd(r,cH,d)=1$. Then $M_H(v)$ is the fine moduli space by \cite[Proposition 10.20]{huybre}. Let $\cU$ be a universal family of the moduli space $M_H(v)$. Then we have an exact functor
$\Phi_{\cU}\colon \Db(M_H(v)) \to \Db(X)$ such that $\Phi_{\cU}(k([E]))=E$ for any closed point $[E] \in M_H(v)$.
Assume that $v$ is an isotropic vector. Then $M_H(v)$ is also a K3 surface and $\Phi_{\cU}\colon \Db(M_H(v)) \to \Db(X)$ is an equivalence by \cite[Proposition 10.25]{huybre}. By Lemma \ref{lem:FM and supp}, if $r>0$ holds, there is an immersion 
\begin{align*}
    X \sqcup M_H(v) &\xrightarrow{\sim} \bspec\Db(X)\, \sqcup \,\Phi_{\cU}\left(\bspec\Db(M_H(v))\right)\\
    &\hookrightarrow \spec\Db(X).
\end{align*}
\end{exa}

\begin{exa}\label{CY3}
Let $V \defeq \bC^7$ be a $7$-dimensional vector space.
Consider the Pl\"ucker embedding $G(2,V) \hookrightarrow \bP(\wedge^2V)$ and
the Pfaffian variety $\mathrm{Pf}\defeq \{[A] \in \bP(\wedge^2 V^*) \mid \mathrm{rk}(A) \leq 4\}$. 
By \cite[Theorem 0.3]{bc}, we can take a $7$-dimensional linear subspace $L \subset \wedge^2V^*$ such that 
$X\defeq G(2,V) \cap \bP(L^\perp)$ and $Y\defeq \mathrm{Pf} \cap \bP(L)$ are Calabi-Yau $3$-folds.
Then  by \cite[Theorem 6.2]{bc} there is an equivalence $\Phi\colon \Db(Y) \xrightarrow{\sim} \Db(X)$  such that $\Phi(k(y))=I_{C_y}$ holds for any closed point $y \in Y$, where  $I_{C_y}$ is the ideal sheaf of some curve $C_y$ in $X$. By Lemma \ref{lem:FM and supp}, there is an immersion
\[X \sqcup Y \xrightarrow{\sim} \bspec\Db(X) \sqcup \Phi\left(\bspec\Db(Y)\right) \hookrightarrow \spec\Db(X).\]
\end{exa}

In contrast to Lemma \ref{lem:FM and supp}, we will see that some fully faithful functors arising from birational geometry induces  non-empty intersections of Balmer spectrums of smooth varieties. 

Let $X$ and $Y$ be smooth varieties, and let $U_X\subseteq X$ and $U_Y\subseteq Y$ are open subvarieties such that there exists an isomorphism $\varphi\colon U_Y\simto U_X$. Denote by  $\pi_{U_X}\colon\Db(X)\to\Db(U_X)$ and $\pi_{U_Y}\colon \Db(Y)\to\Db(U_Y)$  the restriction functors.

\begin{lem}\label{non-empty}
 Let $\Phi\colon\Db(X)\hookto \Db(Y)$ be a fully faithful  functor that admits a right adjoint functor $\Psi\colon \Db(Y)\to \Db(X)$. Assume that the following diagram commutes 
\[
\begin{tikzcd}
 \Db(Y)\arrow[rr,"\Psi"]\arrow[d,"\pi_{U_Y}"']&&\Db(X)\arrow[d,"\pi_{U_X}"]\\
 \Db(U_Y)\arrow[rr,"\varphi_*(-)\otimes L"]&&\Db(U_X)
 \end{tikzcd}
\]
for some $L\in \Pic U_X$. Then we have 
\[
\pi_{U_Y}^{-1}(\bspec\Db(U_Y))\subseteq \Phi_*(\bspec\Db(X))\cap \bspec\Db(Y).
\]

\begin{proof}
By Lemma \ref{open commute}, we have an inclusion $\pi_{U_Y}^{-1}(\bspec\Db(U_Y))\subset \bspec\Db(Y)$. For  $x\in U_X$ and $F\in \Db(Y)$, we have
\begin{eqnarray*}
F\in \pi_{U_Y}^{-1}(\cS_{U_Y}(\varphi(x)))&\Longleftrightarrow& F_{\varphi(x)}\cong 0\\
&\Longleftrightarrow& (\varphi_*(F|_{U_Y})\otimes L)_x\cong 0\\
&\Longleftrightarrow& \Psi(F)_x\cong 0\\
&\Longleftrightarrow& F\in \Phi_*(\cS_X(x)).
\end{eqnarray*}
This shows the equality
$
\pi_{U_Y}^{-1}(\cS_{U_Y}(\varphi(x)))=\Phi_*(\cS_X(x)).
$
Hence we obtain the other inclusion $\pi_{U_Y}^{-1}(\bspec\Db(U_Y))\subset \Phi_*(\bspec\Db(X))$.
\end{proof}
\end{lem}

\begin{exa}[Blow-up]\label{blow up}
Let $X$ be a smooth variety, and let us denote by $f\colon \widetilde{X}\to X$ the blow-up of $X$ along a smooth subvariety $Y\subset X$ of codimension greater than one. Then the functor 
\[
f^*\colon \Db(X)\to \Db(\widetilde{X})
\]
is fully faithful, and it admits a right adjoint functor
 $f_*\colon \Db(\widetilde{X})\to \Db(X)$. If we set  $U\defeq X\backslash Y$ and $\widetilde{U}\defeq \pi^{-1}(U)$, the restriction $f|_{\widetilde{U}}\colon \widetilde{U}\to U$ of $f$ is an isomorphism, and  we have an isomorphism $
 \pi_{U}\circ f_*\cong (f|_{\widetilde{U}})_*\circ \pi_{\widetilde{U}}
 $
of functors. Then by Lemma \ref{non-empty}, we have 
\[
\pi_{\widetilde{U}}^{-1}(\bspec\Db(\widetilde{U}))\subseteq (f^*)_*(\bspec \Db(X))\cap \bspec\Db(\widetilde{X}). 
\]
If $y\in Y$ is a closed point, by Lemma \ref{closed points} we see that  $(f^*)_*(\cS_X(y))\notin \bspec\Db(\widetilde{X})$.
Hence we see that \[(f^*)_*(\bspec \Db(X))\not\subseteq \bspec\Db(\widetilde{X}).\]
\end{exa}

\begin{exa}[Standard flip/flop]\label{standard flip} Let $X$ and $Y$ be a smooth varieties, and $X\xleftarrow{q}W\xrightarrow{p}Y$ a standard flip (or flop) in the sense of  \cite[Section 11.3]{huybre}. Then $q$ and $p$ are the blow-ups of $X$ and $Y$ along some smooth subvarieties $Z_X\subseteq X$ and $Z_{Y}\subseteq Y$ respectively, which are isomorphic to the projective spaces, and we have $q^{-1}(U_X)=p^{-1}(U_Y)$, where $U_X\defeq X\backslash Z_X$ and $U_Y\defeq Y\backslash Z_Y$. Moreover, the functor 
\[
\Phi\defeq p_*\circ q^*\colon \Db(X)\to \Db(Y)
\]
is fully faithful \cite[Theorem 3.6]{bo}, and it admits a right adjoint functor $\Psi\defeq q_*\circ p^!\colon \Db(Y)\to \Db(X)$. Write $q_U\colon q^{-1}(U_X)\to U_X$ and $p_U\colon p^{-1}(U_Y)\to U_Y$ for the isomorphisms defined as the restrictions of $q$ and $p$, and define $\varphi\defeq q_U\circ (p_U)^{-1}\colon U_Y\simto U_X$. Then we have an isomorphism 
$
\pi_{U_X}\circ \Psi\cong \varphi_*\circ \pi_{U_Y}
$
of functors, where $\pi_{U_X}$ and $\pi_{U_Y}$ are the restrictions. Hence by Lemma \ref{non-empty}, we have 
\[
\pi_{U_Y}^{-1}(\bspec\Db(U_Y))\subseteq \Phi_*(\bspec\Db(X))\cap\bspec\Db(Y).
\]
If $z\in Z_X$ is a closed point, by Lemma \ref{closed points} we see that  $(\Phi)_*(\cS_X(z))\notin \bspec\Db(Y)$.
Hence we have 
\[
(\Phi)_*(\bspec \Db(X))\not\subseteq \bspec\Db(Y).
\]
\end{exa}

The following shows that  the subspace $\bspec \Db(X)\subset \spec \Db(X)$ is not closed in general. 
\begin{rem}
Let $X$ and $Y$ be same as in Example \ref{standard flip}. Then $\bspec\Db(X)$ is not closed in $\spec\Db(X)$.
\begin{proof} 
 If the subspace $\bspec \Db(X)\subseteq\spec\Db(X)$ is closed, then the intersection $\Phi_*(\bspec\Db(X))\cap\bspec\Db(Y)$ is closed in $\bspec\Db(Y)$. Hence  we have 
 \begin{equation}\label{remark 3.10}
 \pi_{U_Y}^{-1}(\bspec\Db(U_Y))=\Phi_*(\bspec\Db(X))\cap\bspec\Db(Y)
 \end{equation}
 since  $(\Phi)_*(\cS_X(z))\notin \bspec\Db(Y)$ for any closed point $z\in Z_X$ and \[\left(\Phi_*(\bspec\Db(X))\cap\bspec\Db(Y)\right)\backslash \pi_{U_Y}^{-1}(\bspec\Db(U_Y))\] is closed in $\bspec\Db(Y)$.
 But \eqref{remark 3.10} shows that $\pi_{U_Y}^{-1}(\bspec\Db(U_Y))$ is open and closed in $\bspec\Db(Y)$, which is a  contradiction.
\end{proof}
\end{rem}

\subsection{Maximal thick subcategories}

In this section, we  provide some examples of maximal thick subcategories.

\begin{dfn}
Let $\cT$ be a triangulated category. 
\begin{itemize}
\item[(1)] A thick subcategory $\cM$ of $\cT$ is said to be {\it maximal} if $\cM\neq \cT$ and there is no thick subcategory $\cN$  such that $\cM\subsetneq \cN\subsetneq \cT$.
\item[(2)] A thick subcategory $\cU$ of $\cT$ is said to be {\it proper} if $\cU\neq 0, \cT$.
\item[(3)] A triangulated category is said to be {\it simple} if it has no proper thick subcategories. 
\end{itemize}
\end{dfn}

By definition maximal thick subcategories are automatically prime thick subcategories.  We denote by \[\Max(\cT)\subseteq\spec\cT\] the set of maximal thick subcategories of $\cT$. 

\begin{exa}\label{example:simple}
The bounded derived category $\Db(K)$ of a filed $K$ is simple. 
\end{exa}

\begin{rem}\label{rem:max simple}
 Note that $\cM\in \Th(\cT)$ is maximal if and only if the quotient $\cT/\cM$ is simple by \cite[Lemma 3.1]{takahashi2}.  
\end{rem}

\begin{prop}\label{orthogonal prime}
Let $E\in \cT$ be an exceptional object. Then the left orthogonal complement
\[
^{\perp}E\defeq\{A\in \cT\mid \Hom(A,E[n])=0 \mbox{ for all } n\in \bZ\}
\]
is a maximal thick subcategory of $\cT$. Similarly, the right orthogonal complement $E^{\perp}$ is also a maximal thick subcategory of $\cT$.
\begin{proof}
The quotient $\cT/^{\perp}E$ is equivalent to $\l E\r\cong \Db(k)$. Since $\Db(k)$ is simple, the left orthogonal $^{\perp}E$ is maximal by Remark \ref{rem:max simple}.
\end{proof}
\end{prop}

The following shows that for any integral noetherian scheme $X$, $\Perf X$ always has a maximal thick subcategory.

\begin{prop}
Let $X$ be an integral noetherian scheme, and $\xi\in X$ the generic point of $X$. Then $\cS_X(\xi)$ is a maximal thick subcategory of $\Perf X$.
\begin{proof}
Let $\xi\in U$ be an affine open subset, and set $Z\defeq X\backslash U$. By definition we have $\Perf_ZX\subseteq \cS_X(\xi)$. 
 By \eqref{eq:up to direct} and \[\Th(\Perf X/\Perf_ZX)=\{\cU\in \Th(\Perf X)\mid \Perf_ZX\subseteq \cU\}\] we have an order-preserving bijection
\[
\Omega\colon \Th(\Perf U)\simto \{\cU\in \Th(\Perf X)\mid \Perf_ZX\subseteq \cU\},
\]
and  we have $\Omega(\cS_U(\xi))=\cS_X(\xi)$. Hence it is enough to show that $\cS_U(\xi)$ is maximal in $\Perf U$. Since $\Perf U/\cS_U(\xi)\cong \Perf(\cO_{U,\xi})$ by \cite[Lemma 3.2 (1)]{matsui} and $\cO_{U,\xi}$ is the function field $K(X)$ of $X$, the thick subcategory $\cS_U(\xi)\subseteq \Perf U$ is maximal by Example \ref{example:simple} and Remark \ref{rem:max simple}.
\end{proof}
\end{prop}

We propose the following naive question.

\begin{ques}\label{ques:1}
Let $X$ be a smooth variety. For any $\cP\in \spec(\Db(X))$, is there a maximal thick subcategory $\cM\in \Max(\Db(X))$ such that $\cP\subseteq\cM$?
\end{ques}

By \cite[Proposition 2.3]{mt}, for any $\cP\in \spec\Db(X)$ we have 
\[
\overline{\{\cP\}}=\{\cQ\in\spec\Db(X)\mid \cQ\subseteq\cP\}.
\]
 This implies that Question \ref{ques:1} is equivalent to the following question.
 
\begin{ques}
Let $X$ be a smooth variety. Does the following equality hold?
\[
\spec\Db(X)=\bigcup_{\cM\in\Max(\Db(X))}\overline{\{\cM\}}.
\]
\end{ques}

\begin{rem}
 If $X$ is a smooth quasi-affine variety or a smooth projective curve of genus $g\leq 1$, the answers to the above questions are affirmative by Theorem \ref{matsui thm}, \cite[Example 4.10]{matsui} and Theorem \ref{thm:elliptic curve} in the next section. \footnote{After this paper was accepted for the publication, we note that the answers to the above questions are affirmative. We will discuss this problem in a future work.}
\end{rem} 

%%%%%%%%%%%%%%%%%%%%%%%%%%%%%%%%%%%%%%%%%%

\section{Prime thick subcategories on elliptic curves}

\subsection{Derived categories of curves}

Throughout this subsection, $X$ denotes a smooth projective curve. The following propositions are well-known.
\begin{prop}\label{prop:torsion decomposition}
For a coherent sheaf $E$ on $X$, there are a torsion sheaf $T$ on $X$ and a locally free sheaf $F$ on $X$ such that 
$E \cong T \oplus F$.
\begin{proof}
Let $T\subseteq E$ be the torsion subsheaf of $E$, and then $F\defeq E/T$ is  torsion free. Since $X$ is a smooth curve,  the torsion free sheaf $F$ is automatically a locally free sheaf. Now we have an exact sequence
\[
0\to T\to E\to F\to0.
\]
It is enough to show that this sequence splits. By the Serre duality, we have
\[
\Ext^1(F,T)\cong \Hom(T,F\otimes \omega_X).
\]
But $\Hom(T,F\otimes \omega_X)=0$, since the image of a morphism $f\colon T\to F\otimes \omega_X$ is also torsion and $F\otimes \omega_X$ is torsion free. 
\end{proof}
\end{prop}

\begin{prop}[{\cite[Corollary 3.15]{huybre}}]\label{prop:hom dim one}
For any object $E \in \Db(X)$, there are coherent sheaves $E_1, \cdots, E_n$ on $X$ and integers $l_1, \cdots, l_n \in \bZ$ such that $E \cong \bigoplus_{i=1}^{n}E_i[l_i]$.
\end{prop}

We recall some results in \cite{el}, which we need later. For this, we recall the definition of torsion/torsion-free subcategories.

\begin{dfn}[{\cite[Definition 4.3]{el}}]\label{def:torsion}
A proper thick subcategory $\cC$ of $\Db(X)$ is {\it torsion} (resp. {\it torsion free}) if any object $E \in \cC\backslash \{0\}$ is isomorphic to the direct sums of shifts of torsion sheaves (resp. torsion free sheaves).
\end{dfn}

\begin{prop}[{\cite[Corollarry 4.2]{el}}]\label{prop:Elagin-Lunts}
A proper thick subcategory 
$\cC$ of $\Db(X)$ is always torsion or torsion free.
\end{prop}

The following is an easy observation.

\begin{lem}\label{lem:torsion is ideal}
Let $\cC$ be a torsion thick subcategory of $\Db(X)$. Then $\cC$ is an ideal of $(\Db(X), \otimes_X)$.
\begin{proof}
Let $L\in\Pic(X)$ be a very ample line bundle. By \cite[Theorem 4]{orlov}, the smallest triangulated subcategory of $\Db(X)$ containing $ \cO_X\oplus L$ is $\Db(X)$ itself. Thus it is enough to show that $L\otimes \cC\subseteq\cC$. By Proposition \ref{prop:hom dim one}, we only need  to prove that for any indecomposable torsion coherent sheaf $F\in\cC\cap \coh X$ we have $L\otimes F\in \cC$.   If we denote by $Z$ the scheme theoretic support  of $F$,  there is a unique coherent sheaf $F_0$ on $Z$ such that $F\cong j_*(F_0)$, where $j\colon Z\hookto X$ is the natural closed immersion. Since $F$ is an indecomposable torsion sheaf, its (set theoretic) support $\Supp(F)$ consists of a unique point $x$, and so $Z\cong \Spec R$ for some local ring $R$ of Krull dimension zero. Hence the Picard group $\Pic(Z)$ is trivial, and in particular we have $j^*L\cong \cO_Z$. Now $L\otimes F\in \cC$ follows from  the   isomorphisms
$
L\otimes F\cong L\otimes j_*(F_0)\cong j_*(j^*(L)\otimes F_0)\cong F\in \cC.
$
\end{proof}
\end{lem}

The following results are  immediate consequences of Lemma \ref{lem:torsion is ideal}.

\begin{cor}\label{cor:torsion}
Let $\cP \in \spec \Db(X)$ be a prime thick subcategory. 
Then $\cP$ is torsion if and only if $\cP \in \bspec \Db(X)$ holds.
\begin{proof}
If $\cP$ is torsion, $\cP$ is a prime ideal of $(\Db(X), \otimes_X)$ by Theorem \ref{matsui thm} and Lemma \ref{lem:torsion is ideal}. If $\cP \in \bspec \Db(X)$ holds, $\cP$ is torsion by Remark \ref{rem:supp}. 
\end{proof}
\end{cor}

\begin{cor}
The subspace $\bspec\Db(X)$ is open and closed in $\spec\Db(X)$. 
\begin{proof}
By Proposition \ref{closure}, the closure of $\bspec\Db(X)$ is $\Z(\cE)$, where $\cE=\{E\in\Db(X)\mid \Supp E=X\}$. If $\cP\in \Z(\cE)$, then $\cP$ must be torsion, and so $\cP$  lies in $\bspec\Db(X)$ by Lemma \ref{lem:torsion is ideal}. This proves that $\bspec\Db(X)$ is closed in $\spec\Db(X)$. 

Again by Lemma \ref{lem:torsion is ideal}, if we set $\cF\defeq\{F \in \Db(X)\mid \Supp F\neq X, F \neq 0\}$, we have  $\spec\Db(X)\backslash\bspec\Db(X)\subseteq \Z(\cF)$. The opposite inclusion $\Z(\cF)\subseteq\spec\Db(X)\backslash\bspec\Db(X) $ follows from Remark \ref{rem:supp}. Hence we have  \[\spec\Db(X)\backslash\bspec\Db(X)=\Z(\cF),\] which shows that $\bspec\Db(X)$ is open in $\spec\Db(X)$.
\end{proof}
\end{cor}

\subsection{Moduli spaces of stable sheaves}

In this subsection, following \cite{hp}, we recall basic facts about moduli spaces of semistable sheaves on elliptic curves.

Let $X$ be an elliptic curve.  
 The following is standard.

\begin{prop}[{\cite[Lemma 1]{hp}}]\label{prop:stable}
Let $E$ be a locally free sheaf on $X$. Consider the following conditions.
\begin{itemize}
    \item[$(1)$] The sheaf $E$ is $\mu$-stable. 
    \item[$(2)$] The sheaf $E$ is indecomposable.
    \item[$(3)$] The sheaf $E$ is $\mu$-semistable.
\end{itemize}
Then $(1) \Rightarrow (2)$ and $(2) \Rightarrow (3)$ hold.
If $\mathrm{rk}(E)$ and $\deg(E)$ are coprime, then $(3) \Rightarrow (1)$ holds.
\end{prop}

We define the set $I\defeq \{(r,d) \in \bZ^2 \mid r>0, \gcd(r,d)=1  \}$. 
The set $I$ parametrizes the fine moduli space of $\mu$-semistable sheaves on $X$.

\begin{prop}[{\cite[Lemma 2, Remark, Proposition 3]{hp}}]\label{prop:fine moduli}
Take $(r,d) \in \bZ^2$ such that $r$ is positive.
Let $M(r,d)$ be the moduli space of $\mu$-semistable sheaves with chern character $(r,d)$. Denote the $\mu$-stable locus of $M(r,d)$ by $M^{\mathrm{st}}(r,d)$. Then $M(r,d)=M^{\mathrm{st}}(r,d)\neq \emptyset$ holds if and only if $(r,d) \in I$ holds. 
Moreover, if $(r,d) \in I$ holds, $M(r,d)$ is the fine moduli space.
\end{prop}
Take $(r,d) \in I$. By Proposition \ref{prop:fine moduli},
there is  a universal family $\cU_{r,d}$ of the moduli space $M(r,d)$.  
If $\cU'_{r,d}$ is another universal family of the moduli space $M(r,d)$ there is an invertible sheaf $\cL$ on $M(r,d)$ such that $\cU'_{r,d} \cong \cU_{r,d} \otimes p_{M(r,d)}^*\cL$, where $p_{M(r,d)} \colon M(r,d) \times X \to M(r,d)$ is the projection.
The Fourier-Mukai functor $\Phi_{\cU_{r,d}}\colon \Db(M(r,d)) \to \Db(X)$ gives an equivalence and $\Phi_{\cU_{r,d}}(k([E]))=E$ holds for any closed point
$[E] \in M(r,d)$.
%By \cite{Huy}, $M(r,d)$ is isomorphic to $X$. %Take an isomorphism $\xi_{r,d}:M(r,d) \xrightarrow{\sim} X$. 
%We consider the autoequivalence 
%$\Phi_{r,d}:=\Phi_{\cU_{r,d}} \circ \xi_{r,d}^* : \Db(X) \xrightarrow{\sim} \Db(X)$. 
%Indecomposable locally free sheaves on $X$

\begin{prop}[{\cite[Proposition 4]{hp}}]\label{prop:FM transform}
Let $E$ be a $\mu$-semistable locally free sheaf on $X$.
Consider 
\[(r,d) \defeq \frac{1}{\gcd{(\mathrm{rk}(E),\deg(E))}}(\mathrm{rk}(E), \deg(E)) \in I.\]
Then the object $\Phi^{-1}_{\cU_{r,d}}(E)$ is a torsion sheaf on $M(r,d)$.
\end{prop}

\subsection{Matsui spectrum for elliptic curves}
In this section, we describe the Matsui spectrum of the derived category of an elliptic curve. 

Let $X$ be an elliptic curve.
Take an element $(r,d) \in I$. By Theorem \ref{matsui thm}, we have the immersion $\cS_{M(r,d)}\colon M(r,d) \hookrightarrow \spec \Db(M(r,d))$.  %By Proposition \ref{prop Balmer}, the image $\cS_{M(r,d)}(M(r,d))$ is equal to $\Spec_{\otimes_{\cO_{M(r,d}}}(\Db(M(r,d))$. 
Consider an immersion $i_{r,d}\colon M(r,d) \hookrightarrow \spec \Db(X)$ defined by the composition $i_{r,d}\defeq \Phi_{\cU_{r,d}} \circ \cS_{M(r,d)}$, and set  $M_{r,d}\defeq i_{r,d}(M(r,d))$. By Proposition \ref{prop Balmer}, we have $M_{r,d}=\Phi_{\cU_{r,d}}(\bspec \Db(M(r,d)))$.
Thus the image $M_{r,d}$ is independent to a choice of a universal family $\cU_{r,d}$.

%where $\cS_{M(r,d)}: M(r,d) \hookrightarrow \Spec_\Delta(\Db(M(r,d)))$ and $(\Phi_{\cU_{r,d}})_*: \Spec_\Delta(\Db(M(r,d))) \xrightarrow{\sim}\Spec_\Delta(\Db(M(r,d)))$. 
%See Theorem \ref{matsui thm} (2). 
The main theorem in this subsection is the following.
\begin{thm}\label{thm:elliptic curve}
We have the equality
\[\spec \Db(X)=\bspec\Db(X) \sqcup \bigsqcup_{(r,d) \in I} M_{r,d}. \]
\end{thm}

First, we prove the following lemma.

\begin{lem}\label{lem:union}
We have the equality
\[\spec \Db(X)=\bspec \Db(X) \cup \bigcup_{(r,d) \in I} M_{r,d}. \]
\end{lem}
\begin{proof}
Take a prime thick subcategory $\cP \in \spec\Db(X)$.
By Proposition \ref{prop:Elagin-Lunts}, $\cP$ is torsion or torsion free subcategory of $\Db(X)$.
Assume that $\cP$ is torsion. By Lemma \ref{lem:torsion is ideal}, $\cP$ is an ideal.
By Theorem \ref{matsui thm}, we have $\cP \in \bspec \Db(X)$.
Assume that $\cP$ is torsion free.
By Proposition \ref{prop:torsion decomposition} and Proposition  \ref{prop:hom dim one}, we can take an object $E \in \cP$ such that $E$ is an indecomposable locally free sheaf on $X$. By Proposition \ref{prop:stable}, $E$ is $\mu$-semistable.
If we consider
\[(r,d)\defeq \frac{1}{\gcd{(\mathrm{rk}(E),\deg(E))}}(\mathrm{rk}(E), \deg(E)) \in I,\]
by Proposition \ref{prop:FM transform}, $\Phi^{-1}_{\cU_{r,d}}(E)$ is a torsion sheaf on $M(r,d)$.
Since $\Phi^{-1}_{\cU_{r,d}}(E) \in \Phi^{-1}_{\cU_{r,d}}(\cP)$, the prime thick subcategory $\Phi^{-1}_{\cU_{r,d}}(\cP)$ is torsion by Proposition \ref{prop:Elagin-Lunts}.
By Lemma \ref{lem:torsion is ideal}, we have $\Phi^{-1}_{\cU_{r,d}}(\cP) \in \bspec \Db(M(r,d))$. Therefore, we obtain $\cP \in M_{r,d}$.
\end{proof}

To prove Theorem \ref{thm:elliptic curve}, it is enough to prove the following lemma.

\begin{lem}\label{lem:disjoint}
The followings hold.
\begin{itemize}
    \item[$(1)$]For an element $(r,d) \in I$, we have $\bspec \Db(X) \cap M_{r,d} = \emptyset$.
    \item[$(2)$]For distinct elements $(r_1,d_1)\neq (r_2,d_2) \in I$, we have $M_{r_1, d_1} \cap M_{r_2, d_2}=\emptyset$.
\end{itemize}
\end{lem}
\begin{proof} 
(1) Take $(r,d) \in I$ and $\cP \in M_{r,d}$. Since $M_{r,d}=\Phi_{\cU_{r,d}}(\bspec\Db(M_{r,d}))$ and $\Supp(\Phi_{\cU_{r,d}}(k([E])))=X$ for any closed point $[E]\in M_{r,d}$, this follows from Lemma  \ref{lem:FM and supp}.

(2) Take distinct elements $(r_1,d_1)\neq (r_2,d_2) \in I$. Assume that there is an element $\cP\in M_{r_1, d_1} \cap M_{r_2, d_2}$.
By (1) and  Corollary \ref{cor:torsion}, $\cP$ is a torsion free subcategory of $\Db(X)$. For $i=1,2$,
the prime thick subcategory $\Phi^{-1}_{\cU_{r_i,d_i}}(\cP)$ lies in $\bspec\Db(M(r_i,d_i))$. 
So there exist a point $x_i \in M(r_i,d_i)$ such that $\Phi^{-1}_{\cU_{r_i,d_i}}(\cP)=\cS_{M(r_i,d_i)}(x_i)$. 
Note that $\cS_{M(r_i,d_i)}(x_i)$ is classically generated by \[
\left\{k([E])\in \cS_{M(r_i,d_i)}(x_i) \middle|  \mbox{  $[E]\in M(r_i,d_i)\backslash\{x_i\}$ is a closed point} \right\}.
\] 
Since $\Phi_{\cU_{r_i,d_i}}(k([E]))=E$ for any closed point $[E]\in M(r_i,d_i)$,   we obtain 
\[\{(\mathrm{rk}(E),\deg(E)) \in \bZ^2 \mid E \in \cP \}=\{n(r_i,d_i) \in \bZ^2 \mid n \in \bZ \}.\]
However  we have
\[\{n(r_1,d_1) \in \bZ^2 \mid n \in \bZ \} \cap \{n(r_2,d_2) \in \bZ^2 \mid n \in \bZ \}=\{(0,0)\}.\]
This is a contradiction.
\end{proof}

By Lemma \ref{lem:union} and Lemma \ref{lem:disjoint}, we have proved Theorem \ref{thm:elliptic curve}.

\begin{cor}
Let $\cP \in \spec \Db(X)$ be a prime thick subcategory.
The prime thick subcategory $\cP$ is torsion free if and only if $\cP \in M_{r,d}$ holds for some element $(r,d) \in I$.
\end{cor}
\begin{proof}
%The first statement is proved in the proof of Lemma \ref{lem:union}.
This follows from Corollary \ref{cor:torsion} and Theorem \ref{thm:elliptic curve}.
\end{proof}

\begin{rem}
Matsui pointed out that we can describe the partially ordered set $\mathbf{Th}(\Db(X))$ by the identical argument of Theorem \ref{thm:elliptic curve}. Let $\mathbf{Th}_{\otimes}(X)$ denote the set of all ideals of $(\Db(X), \otimes_X)$.
Then we have the equality 
\[ \mathbf{Th}(\Db(X))=\mathbf{Th}_{\otimes}(X) \sqcup \bigsqcup_{(r,d) \in I} \mathbf{Th}_{\otimes}(M(r,d)) \]
as in Theorem \ref{thm:elliptic curve}.
Taking supports, we have an order preserving bijection between $\mathbf{Th}_\otimes(X)$ (resp. $\mathbf{Th}_\otimes(M(r,d))$) and the partially ordered set of all specialization closed subsets of $X$ (resp. $M(r,d)$) by \cite[Theorem 4.10]{balmer}. 
Note that all ideals of $(\Db(X), \otimes_X)$ are radical by \cite[Proposition 4.4]{balmer}.

\end{rem}

%%%%%%%%%%%%%%%%%%%%%%%%%%%%%%%%%%%%%%%%%%%%%%%%%%%

\section{Serre invariant prime thick subcategories}
In this section, we define the Serre invariant locus of Matsui spectrum, and discuss its basic properties.

Let $\cT$ be a triangulated category  with finite dimensional Hom-spaces. 
A {\it Serre functor} $\bS\colon \cT\simto \cT$ of $\cT$ is an exact autoequivalence such that for any objects $A,B\in \cT$ there is a functorial isomorphism
\[
\Hom_{\cT}(A,B)\cong \Hom_{\cT}(B,\bS(A))^*,
\]
of vector spaces, where $(-)^*$ denotes the dual of a vector space over $k$. It is standard that  Serre functors, if they exist, are unique up to functor isomorphisms, and commute with arbitrary $k$-linear equivalences (see \cite[Section 1.1]{huybre}).  In the remaining of this section, we assume that $\cT$ admits a Serre functor $\bS$. 

\begin{dfn}
A prime thick subcategory $\cP$ of $\cT$ is said to be {\it Serre invariant}, if $\bS(\cP)=\cP$.  We denote by 
\[
\Spec_{\triangle}^{\Ser}\cT
\]
the set of Serre invariant prime thick subcategories of $\cT$.
\end{dfn}

Let $X$ be a projective Gorenstein variety of dimension $d$, and  $f\colon X\to\Spec k$ the morphism defining the base. Then it is standard that the functor $\bR f_*\colon \Perf X\to\Db(k)$ admits a right adjoint functor $f^{!}\colon \Db(k)\to \Perf X$, and the complex $f^{!}(k)$ is quasi-isomorphic to a shifted invertible sheaf  $\omega_X[d]$. Furthermore $\Perf X$ admits a Serre functor $\bS_X$, and we have an isomorphism $\bS_X\cong (-)\otimes_{X}\omega_X[d]$ of functors \cite[Lemma 6.6]{ballard}. Thus any prime ideals of $(\Perf X,\otimes_X)$ are Serre invariant, and so we have an inclusion
\[
\Spec_{\otimes}\Perf X\subseteq\Spec_{\triangle}^{\Ser}\Perf X.
\]

\begin{rem}\label{calabi-yau}
 If $X$ is Calabi-Yau, i.e., $\omega_X\cong\cO_X$, then $\bS_X\cong (-)[d]$. Hence, in this case, we have 
 \[
 \Spec_{\triangle}^{\Ser}\Perf X=\spec\Perf X.
 \]
\end{rem}

The following is one of our motivation of considering Serre invariant locus of Matsui spectrum.

\begin{prop}\label{serre trivial}
Let $X$ be a Gorenstein projective variety. Assume that the smallest thick subcategory of $\Perf X$ containing the family $\{\,\omega_X^n\mid n\in \bZ\,\}$ of line bundles is $\Perf X$ itself. Then a Serre invariant prime thick subcategory of $\Perf X$ is a prime ideal of $(\Perf X,\otimes_X)$. In particular, we have a homeomorphism
\[
X\simto \Spec_{\triangle}^{\Ser}\Perf X.
\]
\begin{proof}
It is enough to show that all Serre invariant prime thick subcategories of $\Perf X$ are ideals. Let $\cP\in \sspec\Perf X$. Since $\cP$ is Serre invariant, we see that  $\cP\otimes_X\omega_X=\cP$. In particular, for any $k\in \bZ$ we have $\cP\otimes_X\omega_X^{k}=\cP$, and this shows that $\cP$ is an ideal since $\Perf X$ is generated by the family $\{\,\omega_X^n\mid n\in \bZ\,\}$. 
\end{proof}
\end{prop}

The following is an immediate consequence of Proposition \ref{serre trivial}

\begin{cor}\label{serre trivial cor}
Let $X$ be a Gorenstein projective variety. If $\omega_X$ is ample or anti-ample, we have a homeomorphism
\[
X\simto \Spec_{\triangle}^{\Ser}\Perf X.
\]
\end{cor}

\begin{rem}
 Let $X_1$ and $X_2$ be Gorenstein projective varieties with $\omega_{X_i}$ ample or anti-ample.  Corollary \ref{serre trivial cor} shows that an exact equivalence $\Perf X_1\cong \Perf X_2$ implies that there is a homeomorphism $X_1\simto X_2$.   If $X_1$ and $X_2$ are smooth projective varieties over $\bC$ of dimension greater than one, combination of Corollary \ref{serre trivial cor} and  \cite[Main Theorem B]{klos} shows that an exact equivalence $\Perf X_1\cong \Perf X_2$ implies that there is an isomorphism $X_1\simto X_2$. This recovers the famous Bondal--Orlov's reconstruction theorem \cite{bo}. 
\end{rem}

By Remark \ref{calabi-yau} and  Corollay \ref{serre trivial cor}, we obtain the following.

\begin{cor}
Let $C$ be a smooth projective curve of genus $g$. 
\begin{itemize}
\item[$(1)$] If $g=1$, we have an equality
\[
\sspec\Db(C)=\spec\Db(C).
\]
\item[$(2)$] If $g\neq1$, we have a homeomorphism
\[
C\simto\sspec\Db(C).
\]
\end{itemize} 
\end{cor}

Serre invariant locus is not equal to Matsui spectrum in general. To see this, we prove the following lemmas.

\begin{lem}\label{not ideal}
Let $X$ be a smooth projective variety of dimension $d\geq1$, and let $E\in\Db(X)$ be an exceptional object. Then the prime thick subcategories $^{\perp}E$ and $E^{\perp}$ are both not  ideals of $(\Db(X),\otimes_X)$.
\begin{proof}
If $^{\perp}E$ is an ideal, it is a prime ideal of $(\Db(X),\otimes_X)$ by Theorem \ref{matsui thm}, and by Theorem \ref{prop Balmer} there exists a point $x\in X$ such that $^{\perp}E=\cS_X(x)$. We note the following claim:\vspace{2mm}\\
{\it Claim.} \hspace{1mm}$\Supp E$ is not a point.\vspace{2mm}\\
Indeed, if there is a closed point $p\in X$ such that $\Supp E=\{
p\}$, then  for each $i\in \bZ$ the $i$-th sheaf cohomology $\cH^i(E)$ has a filtration whose factors are isomorphic to the sky scraper sheaf $k(p)$, and thus, in the Grothendieck group $K(X)$, we have $[\cH^i(E)]=n_i[k(p)]$  for some $n_i\geq0$.   If we set $a_i\defeq \sum_{i\in\bZ}(-1)^in_i$, then we have  $[E]=\sum_{i\in\bZ}(-1)^i[\cH^i(E)]=a_i[k(p)]$. In particular, by a standard isomorphism
$\Ext^i_X(k(p),k(p))\cong \gwedge T_{X,p}$ of vector spaces (see for example \cite[Examples 11.9]{huybre}), we see that  
\[
\chi(E,E)=a_i^2\chi(k(p),k(p))=a_i^2\sum_{i=0}^d(-1)^i\begin{pmatrix}d\\ i\end{pmatrix}=0\neq1,
\]
which contradicts to the assumption that $E$ is exceptional.

By the above claim, there exists a closed point $y\in \Supp E$ with $y\neq x$. Then by  Lemma \ref{supp lem} there exists an integer $m$ such that 
\[
\Hom(k(y),E[m])\neq0.
\]
This shows that $k(y)\notin  ^{\perp}\!\!E$, which contradicts to $^{\perp}E=\cS_X(x)$ and $y\neq x$. 
\end{proof}
\end{lem}

\begin{exa}
By Proposition \ref{serre trivial} we have 
\[
\Spec_{\triangle}^{\Ser}\Db(\bP^n)=\Spec_{\otimes}\Db(\bP^n)\cong \bP^n.
\]
 By Lemma \ref{orthogonal prime} and Lemma \ref{not ideal}, the semi-orthogonal decomposition $
\Db(\bP^{n})=\l\cO(m),\cO(m+1),\hdots,\cO(m+n)\r
$
shows that 
\[
\l\cO(m+1),\hdots,\cO(m+n)\r\in \Spec_{\triangle}\Db(\bP^n)\backslash \Spec_{\triangle}^{\Ser}\Db(\bP^n).
\]
In the case when $n=1$, as  explained in \cite[Example 4.10]{matsui}, by results in \cite[Section 4.1]{ks} and \cite[Corollary 4.9]{matsui} we have  homeomorphisms
\[
\Spec_{\triangle}\Db(\bP^1)\cong \Spec_{\triangle}^{\Ser}\Db(\bP^1)\sqcup\left(\bigsqcup_{i\in\bZ}\l\cO(i)\r\right)\cong \bP^1\sqcup \bZ.
\]
\end{exa}

\end{document}